\documentclass[11pt,reqno,english]{amsart} 
\usepackage[a4paper,left=24mm,right=24mm,top=32mm,bottom=30mm]{geometry} 
\usepackage[T1]{fontenc}
\usepackage[utf8]{inputenc}
\usepackage[a4paper]{geometry}
\usepackage{amsthm}
\usepackage{amsmath}
\usepackage{amssymb}
\usepackage{esint}
\usepackage{xargs}[2008/03/08]

\makeatletter
\numberwithin{equation}{section}
\numberwithin{figure}{section}
\usepackage[color=green!40]{todonotes}
\theoremstyle{plain}
\newtheorem{thm}{Theorem}[section]
  \theoremstyle{definition}
  \newtheorem{defn}[thm]{Definition}
  \theoremstyle{remark}
  \newtheorem{rem}[thm]{Remark}
  \theoremstyle{plain}
  \newtheorem{assumption}[thm]{Assumption}
  \theoremstyle{plain}
  
  \theoremstyle{plain}
  \newtheorem{lem}[thm]{Lemma}
  \theoremstyle{remark}
  
  \theoremstyle{plain}

\makeatother

\usepackage{babel}
\usepackage{comment}

\global\long\def\mathS{\mathcal{S}}

          \global\long\def\Rr{\mathbb{R}}      \global\long\def\Nn{\mathbb{N}} \global\long\def\eps{\varepsilon}                 \global\long\def\diver{\text{div}\,}                                                   

 \newcommandx\upplus[1][usedefault, addprefix=\global, 1=]{\stackrel{{\scriptscriptstyle +}}{#1}}                          
     \global\long\def\Q{\boldsymbol{Q}}               

\global\long\def\norm#1{\left\Vert #1\right\Vert }

\global\long\def\U{\boldsymbol{U}}

\global\long\def\R{\Rr}
\global\long\def\banachB{\mathcal{B}}
\global\long\def\Pp{\mathbb{P}}

\global\long\def\Ee{\mathbb{E}}
\global\long\def\mathA{\mathcal{A}}
\global\long\def\mathE{\mathcal{E}}
\global\long\def\mathC{\mathcal{C}}
\global\long\def\mathG{\mathcal{G}}
\global\long\def\mathS{\mathcal{S}}
\global\long\def\cV{\mathcal{V}}
\global\long\def\oR{\overline{R}}
\global\long\def\ooU{\overline{\U}}
\newcommand{\meta}[1]{{\color{red} [#1]}}
\newcommand{\Ito}{It\={o}\ }
\newcommand{\N}{\mathbb{N}}
\begin{document}

\title[]
{Large deviation principle for a stochastic Allen--Cahn equation}

\author[M. Heida]{Martin Heida}
\email{Martin.Heida@math.tu-dortmund.de} 

\author[M. R\"oger]{Matthias R\"oger}
\email{matthias.roeger@tu-dortmund.de} 

\subjclass[2000]{}


\begin{abstract}
In this paper we consider the Allen--Cahn equation perturbed by a stochastic flux term and prove a large deviation principle. Using an associated stochastic flow of diffeomorphisms the equation can be transformed to a parabolic partial differential equation with random coefficients. We use this structure and first provide a large deviation principle for stochastic flows in function spaces with Hölder-continuity in time. Second, we use a continuity argument and deduce a large deviation principle for the stochastic Allen--Cahn equation.
\end{abstract}

\maketitle

\section{Introduction}
The deterministic Allen--Cahn equation
\begin{gather}\label{eq:ac}
  \eps \partial_t u \,=\, \eps \Delta u - \frac{1}{\eps} W'(u)
\end{gather}
is one prominent example of a mesoscopic model for the dynamics of a two-phase system driven by a reduction of surface energy and not conserving the total mass of the phases. Here $W$ denotes a suitable double-well potential with equal minima in $\pm 1$. The two phases correspond to regions where $u$ is close to $+1$ or $-1$, respectively, and are separated by a thin transition layer of approximate width $\eps$. The Allen--Cahn equation \eqref{eq:ac} is characterized as the accelerated $L^2$ gradient flow of the Van der Waals--Cahn--Hilliard energy
\begin{align}
	E_\eps(u) \,=\, \int_U \frac{\eps}{2}|\nabla u|^2 + \frac{1}{\eps}W(u)\,dx. \label{eq:CH-energy}
\end{align}
By the famous Modica--Mortola theorem \cite{MoMo77,Modi87} this energy approximates the perimeter functional as $\eps\to 0$. Besides its motivation from phase separation theory the Allen--Cahn equation is also intensively studied because of its connection to geometric flows: in the sharp interface limit $\eps\to 0$ solutions $u_\eps$ of \eqref{eq:ac} converge to a family of phase indicator functions $u(\cdot,t)$ that move according to mean curvature flow \cite{MoSc90,EvSS92,Ilma93}. 

Stochastic perturbations of \eqref{eq:ac} have been introduced to include for example thermal
effects or any other unresolved degrees of freedom and to describe nucleation and growth phenomena. A perturbation of the Allen--Cahn equation by additive noise leads to the formal stochastic PDE
\begin{gather}\label{eq:sac-add}
  \eps \partial_t u \,=\, \eps \Delta u - \frac{1}{\eps} W'(u) +
  \sigma\zeta,
\end{gather}
where $\sigma>0$ is a small noise intensity parameter and $\zeta$ denotes
a space-time white noise (in higher space dimensions spatially regularised). 
Such type of evolutions were studied in the one-dimensional
case in \cite{Fu95,BdMP95} and in the higher-dimensional case in \cite{Fu99,We08,LioSou98}.
In higher dimensions the
Allen--Cahn equation with space-time white noise is in general not well-posed and the introduction of spatial correlations by a kind of smoothing procedure are necessary. 

In order to better understand the behavior of solutions to \eqref{eq:sac-add} extensions of the Freidlin--Wentzell theory for randomly perturbed dynamical systems have been used to study the small noise limit $\sigma\to 0$ of \eqref{eq:sac-add}. In \cite{FaJo82} for one space dimension, and in \cite{Feng06},\cite{KoRT06} for higher dimensions the Allen--Cahn action functional was identified and leads (for zero spatial correlation length) to
\begin{gather}
 {{\mathcal S}}_\eps(u)\,:=\, \int_0^T\int_{\R^n} {\eps}(\partial_t u)^2
  +\frac{1}{\eps}\big(-\eps\Delta u 
  +\frac{1}{\eps}W^\prime(u)\big)^2\,dx\,dt, \label{def:action-AC}
\end{gather}
see also \cite{HaWe14} for a recent analysis of coupled limits $\sigma\to 0$ and spatial correlation length to zero.

Here we study an alternative stochastic perturbation of the Allen--Cahn equation in the form of a Stratonovich stochastic partial differential equation (SPDE)
\begin{gather}
  d u_\eps\,=\, \Big( \Delta u_\varepsilon -\frac{1}{\eps^2}W'(u_\varepsilon)\Big) dt +
  \nabla u_\varepsilon\cdot X_\sigma(x,\circ dt), \label{eq:sac-flow}
\end{gather}
where $X_\sigma$ is a vectorfield valued Brownian motion. Such an evolution was introduced in \cite{RoeWe13} where the existence of unique H\"older-continuous strong solutions, the tightness of the solutions 
$(u_\eps)_{\eps>0}$ of \eqref{eq:sac-flow}, and the convergence to an
evolution of (random) phase indicator functions 
$u(t,\cdot)\in BV(U)$ has been shown. 

In this contribution our goal is to first take the noise intensity to zero and to consider the limit process $\sigma\to 0$. More precisely we study the large deviation problem associated to \eqref{eq:sac-flow}, where the driving force is given by a vector-field Brownian motion
\begin{align}
	X_{\sigma}(t,x) \,=\, \sqrt{\sigma}\sum_{l=1}^{\infty}\int_{0}^{t}X^{(l)}(s,x)\,\circ dB_{l}(s)+\int_{0}^{t}X^{(0)}(s,x)\, ds,\label{eq:BM-Field}
\end{align}
see below for the precise assumptions on the coefficients. The large deviation theory developed for the stochastic Allen--Cahn equation with additive noise does not apply here. Instead, we exploit the particular structure of \eqref{eq:sac-flow}. Following the approach by Kunita \cite{Kuni97} we consider the Stratonovich flow associated to $-X_\sigma$, that is the solution of the 
stochastic differential equation
\begin{align}
	d\varphi_{s,t}(x)\,= \,&  -X_\sigma(\circ dt, \varphi_{s,t}(x)) \label{eq:flow1}\\
 	\varphi_{s,s}(x)\,= \,& x, \notag
\end{align}
and use the resulting family of diffeomorphism to transform \eqref{eq:sac-flow} into a partial differential equation with random coefficients $R_\varphi$ and $S_\varphi$
\begin{equation}
	\partial_{t}w-R_{\varphi}:D^{2}w-S_{\varphi}\cdot\nabla w-\frac{1}{\eps^{2}}w+\frac{1}{\eps^{2}}w^{3}=0\,,\label{eq:AC-transform1}
\end{equation}
(for the details see Section \ref{sub:A-continuity-result} below). This approach has been used in \cite{RoeWe13} to prove the existence of solution. Here we also take advantage from the same transformation and deduce a large deviation result for \eqref{eq:sac-flow} from a suitable large deviation principle for \eqref{eq:flow1} and a continuity result for the mapping $\varphi\mapsto w$. 

Large deviation principles for stochastic flows have been obtained by Budhiraja, Dupuis and Maroulas \cite{BuDM10} in suitable classes of time-continuous diffeomorphism, see Section \ref{sec:section-4} below. In order to achieve an appropriate continuity result for the mapping $\varphi\mapsto w$ we however need a large deviation in parabolic Hölder spaces. Therefore one key part in our approach is to suitably extend the corresponding results from \cite{BuDM10}.  

The paper is organized as follows. In the next section we fix some notation and state the precise assumptions and main results for large deviations of vector field valued Brownian motions (Theorem \ref{thm:Main-Thm-1}) and of solutions to the stochastic Allen-Cahn equation (Theorem \ref{thm:Main-1}). In Section \ref{sec:Large-Deviation-Principle}, we introduce suitable function spaces and derive some estimates that are crucial for our calculations in the subsequent sections. Section \ref{sec:section-4} provides the proof of Theorem \ref{thm:Main-Thm-1}, while in Section \ref{sec:section-5} we present the proof of Theorem \ref{thm:Main-1}.
\subsection*{Acknowledgement}
This work was partially funded by the DFG-Forschergruppe 718 \emph{Analysis and Stochastics in Complex Physical Systems}. We thank Hendrik Weber for helpful discussions.
\section{Notation and main results}

We first introduce some notation.

Let $\U\subset\Rr^{n}$ be open and bounded with $C^{\infty}$-boundary
and for some fixed time interval $[0,T]$ let $\Q:=[0,T]\times\overline{\U}$.
We denote by $G^{m}$ the set of $C^{m}$-diffeomorphisms on $\Rr^{n}$. Since $\U$
is bounded, the spaces
\begin{eqnarray*}
C_{id}^{m}(\ooU) & := & \left\{ u\in C^{m}(\bar\U;\Rr^{n})\,:\, u|_{\partial\U}= Id\right\} \,,\\
C_{0}^{m}(\ooU) & := & \left\{ u\in C^{m}(\bar\U;\Rr^{n})\,:\, u|_{\partial\U}=\boldsymbol{0}\right\} \,,\\
G_{id}^{m}(\ooU) & := & \{u\in C_{id}^{m}(\ooU) \text{ is a }C^{m}\text{-diffeomorphism}\} \,,
\end{eqnarray*}
equipped with the $C^m(\bar\U;\R^n)$ norm are Banach spaces. 

For a Banach space $\banachB$ we define \[
l_{2}(\banachB):=\left\{ \left(x_{k}\right)_{k\in\Nn}\subset\banachB\,:\,\norm{\left(x_{k}\right)_{k}}_{l_{2}(\banachB)}<\infty\right\} \,,\quad\norm{\left(x_{k}\right)_{k}}_{l_{2}(\banachB)}^{2}:=\sum_{k=1}^{\infty}\norm{x_{k}}_{\banachB}^{2}\]
and $l_{2}:=l_{2}(\Rr)$. Further, for any compact subset $K\subset\Rr^{m}$
and $0<\alpha<1$ we denote by $C^{0,\alpha}(K;\banachB)$ the space of Hölder-continuous functions on $K$ with values in $\banachB$. 

\subsection{Stochastic flows}\label{sec:section-2-1}
Here we follow Kunita \cite{Kuni97} and introduce Brownian motions with a spatial parameter.
Throughout the paper we fix a filtered probability
space $(\Omega,\mathcal{F},\mathbb{P},\left\{ \mathcal{F}_{t}\right\} )$. 

\begin{assumption}
\label{ass:main-assumption}
We assume that we are given numbers $k\in\N$ with $k\geq 4$, $0<\alpha<1$ and two mappings $a,\,b$ such that 
\begin{itemize}
	\item	$a:\,[0,T]\times\Rr^n\times\Rr^n\rightarrow\Rr^{n\times n}$ is $C^{k,\alpha}$ in the second and third component and continuous in time with $a(\cdot,x,y)=0$ whenever $(x,y)\not\in\U\times\U$,
	\item	$b:\,[0,T]\times\Rr^n\rightarrow\Rr^{n}$ is $C^{k,\alpha}$ in space, and bounded and measurable in time with $b(t,x)=0$ for almost all $t$ and $x\not\in\U$.
\end{itemize}
\end{assumption}
In what follows, we fix $\beta<\alpha$ and consider a continous stochastic process $\left\{ X(t)\right\} _{t\ge 0}$ which is a $C^{k,\beta}$-Brownian
motion on $\Rr^n$ with $X(t,x)=0$ on $\Rr^n\backslash\U$ with local characteristics $(a,b)$ as above, in the following sense (see \cite{Kuni97}): 
\begin{itemize}
\item	$X(0)$, $X(t_{i+1})-X(t_{i})$,
		$i=0,1,\dots,m-1$ are independent $C^{k,\beta}(\Rr^n)$ valued random variables
		whenever $0\leq t_{0}<t_{1}<\dots<t_{m}\leq T$,
\item	for each $x\in\Rr^n$,  the random variable $M(t,x):=X(t,x)-\int_{0}^{t}b(r,x)dr$ is a 				continuous martingale,
\item	for the corresponding quadratic variation we have for all $(x,y)\in\Rr^{n\times n}$ that $\left\langle \left\langle M(\cdot,x),M(\cdot,y)\right\rangle \right\rangle _{t}=\int_{0}^{t}a(r,x,y)dr$.
\end{itemize}

According to Kunita \cite{Kuni97} for any local characteristics
$(a,b)$ given as above and for any $0<\beta<\alpha$ such a $C^{k,\beta}$-Brownian motion exists and can be represented in the form 
\begin{align}
	X(t,x) \,=\, \sum_{i=1}^{\infty}\int_{0}^{t}X^{(i)}(r,x)\,\circ dB_{i}(r)+\int_{0}^{t}X^{(0)}(r,x)\, dr\,,
	\label{eq:BmX}
\end{align}
where $\left(B_{l}\right)_{l\in\Nn}$ is a family of i.i.d. Brownian motions and 
\begin{equation}
	\left(X^{(l)}\right)_{l\in\Nn}\subset L^{2}(0,T;C^{3,\beta}(\ooU)). \label{eq:ass-Xk}
\end{equation}

By the above characterization, we find
\begin{equation}
	a(t,x,y) \,=\, \sum_{i=1}^{\infty}X^{(i)}(t,x)\, X^{(i)}(t,y)^{T}\,,\quad b(t,x)=X^{(0)}(t,x)\,		\label{eq:variance-X}
\end{equation}
and  
\begin{equation}
\sup_{x\in\ooU}\int_{0}^{T}\sum_{i=1}^{\infty}\left|X^{(i)}(r,x)\right|^{2}dr\leq T\sup_{t\in [0,T]}\norm{a(t,\cdot)}_{C_{id}^{k,\alpha}(\U)}<\infty.\label{eq:est-X-sqare-a}
\end{equation}
If $a,\,b$ satisfy the above conditions, we find that $X(t,x)=0$ for $x\not\in\U$.

We associate to a $C^{k,\beta}$-Brownian motion $\left\{ X(t)\right\} _{t\ge 0}$ as above the Stratonovich flow $\left(\varphi_{s,t},s\leq t\right)$ and \Ito flow $\left(\phi_{s,t},s\leq t\right)$,
which satisfy the Stratonovich respectively \Ito initial value problem
\begin{align}
	d\varphi_{s,t}(x) & =-X(\circ dt,\varphi_{s,t}(x))\nonumber \\
 	& =-\sum_{i=1}^{\infty}X^{(i)}(t,\varphi_{s,t})\,\circ dB_{i}(t)-X^{(0)}(t,\varphi_{s,t})\, dt\,,\label{eq:def-stratonovich-flow}\\
	\varphi_{s,s}(x) & =x\,.\nonumber 
\end{align}
and
\begin{align}
	d\phi_{s,t}(x) & =-X(dt,\phi_{s,t}(x))\nonumber \\
 	& =-\sum_{i=1}^{\infty}X^{(i)}(t,\phi_{s,t})\, dB_{i}(t)-X^{(0)}(t,\phi_{s,t})\, dt\,,\label{eq:def-ito-flow}\\
	\phi_{s,s}(x) & =x\,.\nonumber 
\end{align}
We remark that $\varphi_{s,t}(x)=\phi_{s,t}(x)=x$ for all $x\not\in\U$ and thus $\varphi_{s,t},\phi_{s,t}\in C_{id}^{k,\beta}(\ooU)$ almost surely. 
\begin{rem}
By \cite[Theorem 4.6.5]{Kuni97} we can assume w.l.o.g. that the flows $\left(\varphi_{s,t},s\leq t\right)$ and $\left(\phi_{s,t},s\leq t\right)$ are stochastic flows of $C^k$-diffeomorphisms in the sence of \cite{Kuni97}.
%
%
\end{rem}

\subsection{$C^{0,\alpha}$-Large Deviation Principle for stochastic flows}
%

We briefly recall the notions of good rate functions and large deviation principle.
\begin{defn}
Let $\mathcal{E}$ be a Polish space. A function $I:\,\mathcal{E}\rightarrow[0,+\infty]$
is called a good rate function on $\mathcal{E}$, if for each $M<\infty$
the sublevel set $\left\{ x\in\mathcal{E}\,:\, I(x)\leq M\right\} $
is a compact subset of $\mathcal{E}$. For every Borel-measurable $A\subset\mathcal{E}$,
we define $I(A):=\inf_{x\in A}I(x)$.
\end{defn}
\begin{defn}
\label{def:LDP}
Let $I$ be a good rate function on $\mathcal{E}$. A
sequence $\left\{ u^{\sigma}\right\}_\sigma>0 $ is said to satisfy the large
deviation principle  (LDP) on $\mathcal{E}$ with good rate function $I$ if the
following large deviation upper and lower bounds hold:
\begin{itemize}
\item For each closed subset $F$ of $\mathcal{E}$, \[
\limsup_{\sigma\rightarrow0}\sigma\log\mathbb{P}(u^{\sigma}\in F)\leq-I(F)\,.\]

\item For each open subset $G$ of $\mathcal{E}$, \[
\liminf_{\sigma\rightarrow0}\sigma\log\mathbb{P}(u^{\sigma}\in G)\geq-I(G)\,.\]

\end{itemize}
\end{defn}
We next describe a suitable large deviation principle for stochastic flows associated to a $C^{k,\alpha}$-Brownian motion. Let $X^{(k)}$ be given as in Assumption \ref{ass:main-assumption}. For any $\sigma>0$ we define $X_\sigma$ as in \eqref{eq:BM-Field}
\begin{equation}
	X_{\sigma}(t,x)=\sqrt{\sigma}\sum_{l=1}^{\infty}\int_{0}^{t}X^{(l)}(s,x)\, dB_{k}(s)+\int_{0}^{t}X^{(0)}(s,x)\, ds\label{eq:general-form-X-sigma}
\end{equation}
and we associate to $X_{\sigma}(t,x)$ the stochastic flows $\phi_{s,t}^{\sigma}$
 and  $\varphi_{s,t}^{\sigma}$
according to \eqref{eq:def-stratonovich-flow} and \eqref{eq:def-ito-flow}.

Next we define for $f\in L^{2}(0,T;l_{2})$ controlled vector fields 
\begin{align}
	b_{f}(t,x) & :=\sum_{l=1}^{\infty}f_{l}(t)X^{(l)}(t,x)+X_{0}(t,x)\,, & X^{0,f}(t,x) & :=\int_{0}^{t}b_{f}(s,x)ds	\label{eq:LDP-X-0-f}
\end{align}
and controlled associated flows $\left(\phi_{t_{0},t}^{0,f}\right)_{t\in(t_{0},T)}$ that are given as the unique
solution of 
\begin{equation}
	\phi_{t_{0},t}^{0,f}(x)=x+\int_{t_{0}}^{t}b_{f}(s,\phi_{t_{0},s}^{0,f}(x))ds\qquad\forall t\in[t_{0},T],\,\forall x\in\ooU. \label{eq:phi-0-f}
\end{equation}
Our first result is a large deviation principle in spaces with Hölder regularity in time. As described in the introduction this extends results from \cite{BuDM10}, where a large deviation principle in spaces of time-continuous functions has been proved.
\begin{thm}
\label{thm:Main-Thm-1} For $(\varphi^{\sigma},X_{\sigma})_{\sigma>0}$
and $(\phi^{\sigma},X_{\sigma})_{\sigma>0}$ defined above and for any
$0<\gamma<\frac{1}{2}$, $\eta_\phi\in(0,1)$ and $\eta_\varphi\in(0,\beta]$ the family $(\varphi^{\sigma},X_{\sigma})_{\sigma>0}$ and $(\phi^{\sigma},X_{\sigma})_{\sigma>0}$
satisfy LDPs in the spaces $C^{0,\gamma}([0,T];C^{k-1,\eta_\varphi}(\ooU)^ 2)$, $C^{0,\gamma}([0,T];C^{k-1,\eta_\phi}(\ooU)^ 2)$ respectively with the good rate function $I_{W}^{\ast}$ defined by  
\begin{equation}
	I_{W}^{\ast}(\varphi,X)=\inf\left\{ \frac{1}{2}\int_{0}^{T}\norm{f(s)}_{l_{2}}^{2}ds\,\,:\,\, f\in L^{2}(0,T;l_{2})\mbox{ s.t. }(\phi^{0,f},X^{0,f})=(\varphi,X)\right\} \,.\label{eq:LDP-rate-function}
\end{equation}
\end{thm}
We will give a proof of this theorem in Section \ref{sec:section-4}.

\subsection{Large deviation principle for the stochastic Allen-Cahn equation \eqref{eq:sac-flow}}
Without loss of generality, we set $\eps=1$ as the original problem can always be reduced to that case using a parabolic rescaling. In the following we choose the standard quartic double-well potential $W(r)=\frac{1}{4}(1-r^2)^2$. We now describe our main result concerning the solutions of the
Stratonovich stochastic Allen-Cahn equation \eqref{eq:sac-flow}, that is
\begin{align}
	u(t,x) & =u_{0}(x)+\int_{0}^{t}\left(\Delta u-W'(u)\right)+\int_{0}^{t}\nabla u(s,x)\cdot X_{\sigma}(\circ ds,x)\,,\label{eq:main-EQ-1}\\
\nabla u\cdot\nu_{\U} & =0\qquad\mbox{on }(0,T)\times\partial\U\,,
\label{eq:main-EQ-2}\end{align}
where $u_{0}$ are fixed, smooth deterministic inital data, and where $X_\sigma$ was defined in \eqref{eq:general-form-X-sigma}. Under the assumptions stated above existence of unique
continuous $C^{3,\beta}(\overline{\U})$-valued semimartingale solutions $u$ to \eqref{eq:main-EQ-1}, \eqref{eq:main-EQ-2} has been shown in \cite[Theorem 4.1]{RoeWe13}. 

For a deterministic control $f \in L^{2}([0,T];l_{2})$, initial data $u_{0}\in C^{3,\beta}(\overline{\U})$ and $b_f$ as defined in \eqref{eq:LDP-X-0-f} we consider the following differential equation for $u\in C([0,T];C^{3,\beta}(\overline{U}))$, 
\begin{equation}
\begin{aligned}
	u(t,.) &=\, u_{0}
	+\int_{0}^{t}\nabla u(s,.)\cdot b_f(s,.)\, ds+
	\int_{0}^{t}\left(\Delta u(s,.)-W'(u(s,.))\right)\,ds \\
\nabla u\cdot\nu_{\U} & =\,0\qquad\mbox{on }(0,T)\times\partial\U\,,
\end{aligned} \label{eq:ctr-u}
\end{equation}
for all $t\in[0,T]$.
\begin{thm}
\label{thm:Main-1}
The family $(u_{\sigma})_{\sigma>0}$ satisfies 
a large deviation principle in $C([0,T];C^{2,\beta}(\ooU))\cap C^{0,\beta}([0,T];C^{1,\beta}(\ooU))$ for $\sigma\downarrow 0$ with good rate function 
\begin{equation}
	\hat{I}(u) \,=\, \inf\left\{ \frac{1}{2}\int_{0}^{T}\norm{f(s)}_{l_{2}}^{2}ds\,\,:\, f\in L^{2}([0,T];l_{2})\text{ satisfies }\eqref{eq:ctr-u}\right\} \label{eq:rate-func-AC}
\end{equation}
\end{thm}
We give the proof of this theorem in Section \ref{sec:section-5}.


\section{\label{sec:Large-Deviation-Principle}Preliminaries}

\subsection{Function spaces }

To obtain suitable continuity properties in Section \ref{sub:A-continuity-result} it is most convenient to work in parabolic Hölder spaces. 
Therefore, we introduce for any bounded subset
$\U\subset\Rr^{d}$ and any $l>0$ the Hölder spaces $H^{l}(\overline{\U})$ of
$[l]$-times continuously differentiable functions with the $[l]$-th
deriative being in $C^{0,l-[l]}(\overline{\U})$. We denote the corresponding norm
by $\left|\cdot\right|_{\U,l}$. Furthermore, we denote $H^{l/2,l}(\Q)$
the set of functions $u$ satisfying $D_{t}^{r}D_{x}^{s}u\in C(\Q)$
for $2r+s\leq l$ and $D_{t}^{r}D_{x}^{s}u$ being $C^{0,l-[l]}$
in space and $C^{0,\frac{1}{2}(l-[l])}$ in time. The corresponding
norm is denoted $\left|\cdot\right|_{\Q,l}$. 

Working in Hölder spaces has the drawback that these spaces are not separable, which causes some additional difficulties in the proof of the large deviation principle for stochastic flows. To circumvent this problems we introduce the following subspaces of $C^{0,\alpha}(K;\banachB)$ for $0<\alpha<1$, $K\subset \R^d$ compact and $\banachB$ a suitable Banach space:
\begin{align*}
	\lambda^{0,\alpha}(K;\banachB)\,&:=\,\left\{ u\in C^{0,\alpha}(K)\,:\,\lim_{\delta\rightarrow0}\sup_{\left|x-y\right|<\delta}	\frac{\left\Vert u(x)-u(y)\right\Vert _{\banachB}}{\left|x-y\right|^{\alpha}}=0\right\}, \\
	\lambda^{m,\alpha}(K;\banachB)\,&:=\, \{u\in C^m(K;\banachB)\,:\, D^m u \in \lambda^{0,\alpha}(K;\banachB)\}.
\end{align*}
Furthermore we define the spaces
\begin{equation}
\begin{gathered}W_{m}:=C([0,T];C^{m}(\ooU))\,,\qquad\hat{W}_{m}:=C([0,T];G^{m}(\ooU))\,,\\
W_{m}^{\gamma}:=\lambda^{0,\gamma}([0,T];\lambda^{m,2\gamma}(\ooU))\,,\qquad\hat{W}_{m}^{\gamma}:=\lambda^{0,\gamma}([0,T];G^{m}(\ooU))\,,\end{gathered}
\label{eq:def-Wm-Rr-n}
\end{equation}
and set $\lambda_{id}^{m,\gamma}(\ooU):=\lambda^{m,\gamma}(\ooU;\Rr^{n})\cap C_{id}^{m}(\ooU)$,
$\lambda_{0}^{m,\gamma}(\ooU):=\lambda^{m,\gamma}(\ooU;\Rr^{n})\cap C_{0}^{m}(\ooU)$
and
\begin{equation}
\begin{gathered}W_{m,id}^{\gamma}(\Q):=\lambda^{\gamma}([0,T];\lambda_{id}^{m,2\gamma}(\ooU))\,,\qquad W_{m,0}^{\gamma}(\Q):=\lambda^{\gamma}([0,T];\lambda_{0}^{m,2\gamma}(\ooU))\,,\\
\hat{W}_{m}(\Q):=C([0,T];G_{id}^{m}(\ooU)).\end{gathered}
\label{eq:def-Wm-Rr-n2}
\end{equation}

\begin{rem}
\label{rem:hoelder-imbeddings}
By \cite[Theorem 1]{BonicFramptonTromba1969LambdaMan}
the space $\lambda^{0,\alpha}(K;\banachB)$ is separable if $\banachB$ is separable.

Given two Banach spaces $\banachB$ and $\banachB'$, such that $\banachB\hookrightarrow\banachB'$ is compact and $\banachB'$ is separable, the embedding $C^{0,\alpha}(K;\banachB)\hookrightarrow\lambda^{\beta}(K;\banachB')$
is compact for $\beta<\alpha$. For any $\gamma<\frac{1}{2}$ the
embedding $C^{0,\gamma}([0,T];C^{0,2\gamma}(\ooU))\hookrightarrow H^{\gamma,2\gamma}(\Q)$
is continuous.
\end{rem}

\subsection{Inequalities and embeddings}

We now derive some useful inequalities and embeddings and start with
a generalization of the Garsia-Rodemich-Rumsey Lemma \cite{garsia1970real}
to Banach-valued functions.
\begin{lem}
\label{lem:Garsia-Rodemich-Rumsey}Let $\banachB$ be a Banach space,
$\alpha\geq0$ and $p\geq1$. There is a constant $C_{\alpha}\geq1$
such that for any $f\in C([0,T];\banachB)$ with the property that
the right hand side in the following inequality is bounded, we have 
$f\in C^{0,\alpha}([0,T];\banachB)$ and
\begin{equation}
\sup_{x,y\in[0,T]}\frac{\norm{f(x)-f(y)}_{\banachB}}{\left|x-y\right|^{\alpha}}\leq C_{\alpha}\left(\int_{[0,T]^{2}}\frac{\norm{f(x)-f(y)}_{\banachB}^{p}}{\left|x-y\right|^{\alpha p+2}}\, dx\, dy\right)^{\frac{1}{p}}\,.\label{eq:sobolev-hoelder}\end{equation}
\end{lem}
\begin{rem}
We can compare \eqref{eq:sobolev-hoelder} to a classical Sobolev
inequality: Considering the space $W^{s,p}(0,T)$ of $\Rr$-valued
functions with the norm \[
\norm u_{s,p}:=\int_{0}^{T}\int_{0}^{T}\frac{\left|u(x)-u(y)\right|^{p}}{\left|x-y\right|^{1+sp}}dx\, dy\,,\]
we find that $W^{s,p}(0,T)\hookrightarrow C^{0,\delta}[0,T]$ in case
$s-1/p\geq\delta$. Compared with \eqref{eq:sobolev-hoelder}, this
corresponds to $s=\frac{1+\alpha p}{p}$ and $\delta=\alpha$.\end{rem}
\begin{proof}
For simplicity, we expand $f$ in a continuous way by constants outside $[0,T]$. We follow the proof of Lemma 4 in \cite{gubinelli2004controlling}.
Consider $\psi(x):=\left|x\right|^{p}$ and $P(x):=\left|x\right|^{\alpha+2/p}$
as functions $\Rr\rightarrow\Rr$. Furthermore, set $R_{xy}:=f(x)-f(y)$.
We then find by convexity of $\psi$ for any measurable sets $A,B\subset[0,T]$
\begin{align}
\int_{A\times B}\norm{R_{xy}}_{\banachB}\frac{dx\, dy}{\left|A\right|\left|B\right|} & \leq P(d(A,B)/4)\psi^{-1}\left(\int_{A\times B}\psi\left(\frac{\norm{R_{xy}}_{\banachB}}{P(d(x,y)/4)}\right)\frac{dx\, dy}{\left|A\right|\left|B\right|}\right)\nonumber\\
 & \leq P(d(A,B)/4)\psi^{-1}\left(\frac{U}{\left|A\right|\left|B\right|}\right)\,,\label{eq:estim-GRR-1}
\end{align}
where $d(A,B)=\sup_{x\in A,y\in B}\left|x-y\right|$ and $U=\int_{[0,T]^{2}}\psi\left(\frac{\norm{R_{xy}}_{\banachB}}{P(d(x,y)/4)}\right)dx\, dy$. 

Let $\oR(t,r_{1},r_{2}):=\int_{B(t,r_{1})}\frac{du}{\left|B(t,r_{1})\right|}\int_{B(t,r_{2})}\frac{dv}{\left|B(t,r_{2})\right|}R_{uv}$ for $r_1,r_2>0$, with $\oR(t,0,r_{2}):=\int_{B(t,r_{2})}\frac{dv}{\left|B(t,r_{2})\right|}R_{tv}$ and $\oR(t,r_{1},0)$ similarly. Note that $\oR$ is continuous on $[0,\infty)^3$ if one sets $\oR(t,0,0)= 0$ for all $t\geq0$. We choose $s,t\in[0,T]$, $s<t$, define $\lambda_{0}:=t-s$ and 
$\lambda_{n+1}$ through $P(\lambda_{n})=2P(\lambda_{n+1})$, inductively. Then,
by monotonicity of $P$, 
\begin{align*}
  P((\lambda_{n}+\lambda_{n+1})/4) & \leq P(\lambda_{n})=2P(\lambda_{n+1})\\
   & = 4P(\lambda_{n+1})-2P(\lambda_{n+1}) = 4\left[P(\lambda_{n+1})-P(\lambda_{n+2})\right]\,.
\end{align*}
We find, using equation \eqref{eq:estim-GRR-1}
\begin{eqnarray*}
\norm{\oR(t,\lambda_{n+1},\lambda_{n})}_{\banachB} & \leq & P((\lambda_{n}+\lambda_{n+1})/4)\psi^{-1}\left(\frac{U}{\lambda_{n}\lambda_{n+1}}\right)\\
 & \leq & 4\left[P(\lambda_{n+1})-P(\lambda_{n+2})\right]\psi^{-1}\left(\frac{U}{\lambda_{n}\lambda_{n+1}}\right)\\
 & \leq & 4\int_{\lambda_{n+2}}^{\lambda_{n+1}}\psi^{-1}\left(\frac{U}{r^{2}}\right)dP(r)\,.
\end{eqnarray*}
For any sequence of variables $\left( x_{i}\right) _{i\in\Nn}\subset\Rr$,
we find \[
R_{x\, x_{0}}=R_{x\, x_{n+1}}+\sum_{i=0}^{n}R_{x_{i+1}\, x_{i}}\,,\]
and averaging with respect to $x_i$ over $B(t,\lambda_i)$ for $i=0,\dots,n+1$ leads to 
\[
\oR(t,0,\lambda_{0})=\oR(t,0,\lambda_{n+1})+\sum_{i=0}^{n}\oR(t,\lambda_{i+1},\lambda_{i})\,.
\]
Since $\oR$ is continuous and $\oR(t,0,0)=0$, taking the limit $n\rightarrow\infty$ yields
\begin{equation}
\norm{\oR(t,0,\lambda_{0})}_{\banachB}\leq\sum_{i=0}^{\infty}4\int_{\lambda_{i+2}}^{\lambda_{i+1}}\psi^{-1}\left(\frac{U}{r^{2}}\right)dP(r)\leq4\int_{0}^{t-s}\psi^{-1}\left(\frac{U}{r^{2}}\right)dP(r)\,.\label{eq:estim-GRR-2}
\end{equation}
Similarly, we find
 \[
\norm{\oR(s,0,\lambda_{0})}_{\banachB}\leq4\int_{0}^{t-s}\psi^{-1}\left(\frac{U}{r^{2}}\right)dP(r)\,,
\]
and a corresponding estimate for $\oR(t,\lambda_0,0)$. We use $R_{s\, t}=R_{s\, x}+R_{x\, y}+R_{y\, t}$,
 and average over the balls $B(s,\lambda_0)$ in $x$ and $B(t,\lambda_0)$ in $y$ to obtain
$$R_{s\, t}=\oR(s,0,\lambda_0)+\int_{B(s,\lambda_0)\times B(t,\lambda_0)}R_{x\,y}\frac{dx\,dy}{4\lambda_0^2}+\oR(t,\lambda_0,0)$$
Note that we can estimate the second term on the right hand side by $P(\frac34)P(\lambda_0)\psi^{-1}(U/\lambda_0^2)$.
Thus, from the last equation together with \eqref {eq:estim-GRR-1} and \eqref {eq:estim-GRR-2} we get \[
\norm{R_{st}}_{\banachB}\leq10\int_{0}^{t-s}\psi^{-1}\left(\frac{U}{r^{2}}\right)dP(r)\,.\]
Using the definition of $\psi$ and $P$, this finally proves the claim.
 \end{proof}
\begin{lem}
\label{lem:Hairer-Maas-Weber}Consider a Banach space $\banachB$,
$p,q\geq1$, $p>2q$ and a random function $f$ with values in $C([0,T];\banachB)$
such that there is a fixed constant $\Lambda$ with \[
\Ee\norm{f(t)-f(s)}_{\banachB}^{p}\leq \Lambda(t-s)^{\frac{p}{2q}}\quad\forall s<t\qquad\mbox{and}\qquad\Ee\norm f_{L^{\infty}(\banachB)}^p\leq C\,.\]
Then, for every $\alpha<\frac{1}{2q}-\frac{1}{p}$ we have \[
\Ee\norm f_{C^{0,\alpha}([0,T];\banachB)}^p\leq C_{\alpha}(\Lambda+1)\,,\]
where $C_{\alpha}$ is a constant depending only on $\alpha,p,q$.\end{lem}
\begin{proof}
By Lemma \ref{lem:Garsia-Rodemich-Rumsey} we find 
\begin{align*}
\Ee\norm f_{C^{\alpha,0}([0,T];\banachB)}^{p} & \leq C_{\alpha}\left(\int_{[0,T]^{2}}\frac{\Ee\norm{f(t)-f(s)}_{\banachB}^{p}}{\left|t-s\right|^{\alpha p+2}}\, ds\, dt\right)+\Ee\norm f_{C([0,T];\banachB)}^{p}\\
 & \leq C_{\alpha}\left(\Lambda\int_{[0,T]^{2}}\frac{\left|t-s\right|^{\frac{p}{2q}}}{\left|t-s\right|^{\alpha p+2}}\, ds\, dt+\Ee\norm f_{L^{\infty}(\banachB)}^{p}\right)
\end{align*}
and the last integral exists iff $\alpha<\frac{1}{2q}-\frac{1}{p}$. 
\end{proof}
We will need a generalized version of the Arzela-Ascoli theorem and
of Kolmogorovs tightness criterium:
\begin{thm}
\label{thm:Arzela-Ascoli-h=0000F6lder}Given two Banach spaces $\banachB$
and $\banachB'$ with $\banachB\hookrightarrow\banachB'$ compactly
and $\banachB'$ separable, let $\left\{ v_{n}\right\} _{n\in\Nn}\subset C([0,T];\banachB)$
be a sequence of functions with $\sup_{n}\sup_{t\in[0,T]}\norm{v_{n}(t)}_{\banachB}<\infty$
and for some $\alpha>0$ let \[
\sup_{n}\norm{v_{n}(t_{1})-v_{n}(t_{2})}_{\banachB}\leq C(t_{1}-t_{2})^{\alpha}\,.\]
Then, $\left\{ v_{n}\right\} _{n\in\Nn}$ is compact in $\lambda^{0,\gamma}([0,T];\banachB')$
and $C^{0,\gamma}([0,T];\banachB')$ for any $\gamma<\alpha$.\end{thm}
\begin{proof}
By the Arzela-Ascoli theorem for Banach space valued continuous functions,
$\left\{ v_{n}\right\} _{n\in\Nn}$ is compact in $C([0,T];\banachB')$.
Furthermore, we find equiboundedness of $\left\{ v_{n}\right\} _{n\in\Nn}$
in $C^{0,\alpha}([0,T];\banachB')$, which yields the desired result
for any $\gamma<\alpha$ through embedding of Hölder spaces and Remark
\ref{rem:hoelder-imbeddings}.\end{proof}
\begin{thm}
\label{thm:Kolmogorov-tightness}Given two Banach spaces $\banachB$
and $\banachB'$ with $\banachB\hookrightarrow\banachB'$ compactly
and $\banachB'$ separable, let $\left\{ \psi_{n}\right\} _{n\in\Nn}$
be a sequence of random fields with values in $C([0,T];\banachB)$.
Assume for any $p>1$ there is a positive constant $C_{p}$ such that
\begin{align}
\Ee\norm{\psi_{n}(t)-\psi_{n}(s)}_{\banachB}^{p} & \leq C_{p}\left|t-s\right|^{\frac{p}{2}}\qquad\forall s,t\in[0,T]\,,\label{eq:Kol-eq-1}\\
\Ee\norm{\psi_{n}(t)}_{\banachB}^{p} & \leq C_{p}\qquad\forall t\in[0,T]\label{eq:Kol-eq-2}\end{align}
for any $n\in\Nn$. Then, $\left\{ \psi_{n}\right\} _{n\in\Nn}$ is
tight in $\lambda^{0,\gamma}([0,T];\banachB')$ for any $\gamma<\frac{1}{2}-\frac{1}{p}$.\end{thm}
\begin{proof}
We follow the proof of \cite{Kuni97} Theorem 1.4.7 For arbitrary $q\in\Nn$, $q>1$, we represent any real number $t$
as $t=\sum_{k=0}^{\infty}a_{k}q^{-k}$, where $a_{k}\in\Nn_{0}$,
$k=0,1,2,\dots$ are non-negative integers and $a_{k}<q$ for all
$k>0$. Let $t_{N}=\sum_{k=0}^{N}a_{k}q^{-k}$ and say $t$ is $q$-adic
of length $N$ if $t=t_{N}$ for some $N$. We introduce $\Delta_{N}$
the set of all $q$-adic rationals of length $N$ and for 
$f\in C([0,T];\banachB)$ the values
\begin{align*}
\Delta_{N}(f) & =\max_{s,t\in\Delta_{N}\,,\,\,\left|s-t\right|\leq q^{-N}}\norm{f(s)-f(t)}_{\banachB}\,,\\
\Delta_{N}^{\gamma}(f) & =\Delta_{N}(f)/\left(q^{-N}\right)^{\gamma}\,.
\end{align*}
We infer from \cite{Kuni97}, Lemmas 1.4.2 and 1.4.3
(note the different meaning of $\gamma$ in this reference) that for
$\gamma<\frac{1}{2}-\frac{1}{p}$, there holds\begin{align*}
\norm{\psi_{n}(s)-\psi_{n}(t)}_{\banachB} & \leq4q\left(\sum_{N=1}^{\infty}\Delta_{N}^{\gamma}(\psi_{n})\right)\left|s-t\right|^{\gamma}\,,\\
\sup_{n}\Ee\left(\left|\sum_{N=1}^{\infty}\Delta_{N}^{\gamma}(\psi_{n})\right|^{p}\right) & <\infty\,.\end{align*}
For any $\eps>0$, Chebyschev's inequality yields existence of $a>0$
such that for all $n\in\Nn$ \[
\begin{aligned}\Pp\left(\sum_{N=1}^{\infty}\Delta_{N}^{\gamma}(\psi_{n})>a\right) & <\frac{\eps}{2}\,,\\
\Pp\left(\norm{\psi_{n}(0)}_{\banachB}>a\right) & <\frac{\eps}{2}\,.\end{aligned}
\]
Let \[
K:=\left\{ f\in C([0,T];\banachB)\,:\,\sum_{N=1}^{\infty}\Delta_{N}^{\gamma}(f)<a\,,\,\norm{f(0)}_{\banachB}<a\right\} \,.\]
If $\gamma<\frac{1}{2}-\frac{1}{p}$, we can repeat the arguments
of Lemma 1.4.2, Lemma 1.4.3 and Theorem 1.4.7 of \cite{Kuni97},
to derive \[
\begin{aligned}\norm{f(t)}_{\banachB} & \leq\norm{f(0)}_{\banachB}+\norm{f(0)-f(t)}_{\banachB}\leq a+4aqt^{\gamma}\,,\\
\norm{f(s)-f(t)}_{\banachB} & \leq4aq\left|s-t\right|^{\gamma}\,\end{aligned}
\]
for all $f\in K$. Since $\gamma<\frac{1}{2}-\frac{1}{p}$ was arbitrary,
$K$ is compact in $\lambda^{0,\gamma}([0,T];\banachB')$ by Theorem
\ref{thm:Arzela-Ascoli-h=0000F6lder}. Finally, note that $\Pp\left(\psi_{n}\not\in K\right)<\eps$
and thus $\psi_{n}$ is tight in $\lambda^{0,\gamma}([0,T];\banachB')$
for all $\gamma<\frac{1}{2}-\frac{1}{p}$.
\end{proof}

\subsection{Large deviation principles and continuous mappings}

For the proof of our main Theorems \ref{thm:Main-Thm-1} and \ref{thm:Main-1} 
we will finally need the following contraction principle.
\begin{thm}[Contraction Principle, \cite{dembo1998large} Theorem 4.2.1]
\label{pro:koenig}Let $\mathcal{E}$ and $\tilde{\mathcal{E}}$
be Hausdorff topological spaces and $F:\,\mathcal{E}\rightarrow\tilde{\mathcal{E}}$
be continuous. If $I$ is a good rate function on $\mathcal{E}$,
the function \[
\tilde{I}(v)=\inf\left\{ I(u)\,:\, v=F(u)\right\} \]
is a good rate function on $\tilde{\mathcal{E}}$. If $\left\{ u_{\sigma}\right\} _{\sigma>0}$
is a sequence of $\mathE$-valued random variables satisfying a large
deviation principle on $\mathcal{E}$ with good rate function $I$, the
sequence $\left\{ F(u_{\sigma})\right\}_{\sigma>0} $ satisfies a large deviation
principle on $\tilde{\mathcal{E}}$ with good rate function $\tilde{I}$.
\end{thm}

\section{\label{sec:section-4}Large deviations for stochastic flows}

The aim of this section is to prove Theorem \ref{thm:Main-Thm-1}. We will obtain this theorem as a consequence of Theorems 
\ref {thm:c-k-gamma-convergence}--\ref{thm:LDP-Flow-2} below. We first introduce some notations.  

Given the filtered probability space 
$(\Omega,\mathcal{F},\mathbb{P},\left\{ \mathcal{F}_{t}\right\} )$ 
from Section \ref{sec:section-2-1}, we define
\begin{align*}
\mathcal{A}\left[l_{2}\right] & :=\left\{ \phi\equiv\left\{ \phi_{i}\right\} _{i\in\Nn}\,\,|\,\,\phi_{i}:[0,T]\rightarrow\Rr\mbox{ is }\left\{ \mathcal{F}_{t}\right\} \mbox{- predictable for all }i\mbox{ and }\right.\\
 & \left.\phantom{:=\{\phi\equiv\left\{ \phi_{i}\right\} _{i\in\Nn}\,\,|\,\,}\int_{0}^{T}\norm{\phi(s)}_{l_{2}}^{2}ds<\infty\mbox{ a.s.}\right\}\,, \\
S_{N}[l_{2}] & :=\left\{ \phi=\left\{ \phi_{i}\right\} _{i\in\Nn}\in L^{2}(0,T;l_{2})\,\,:\,\,\int_{0}^{T}\norm{\phi(s)}_{l_{2}}^{2}ds\leq N\right\}\,,\\
\mathcal{A}_{N}[l_{2}] & :=\left\{ u\in\mathcal{A}[l_{2}]\,\,:\,\, u\in S_{N}[l_2]\mbox{ almost surely}\right\}\,.
\end{align*}
We equip $S_N[l_2]$ with the weak topology in $L^{2}(0,T;l_{2})$ such that $S_{N}[l_{2}]$
is a Polish space.

Now, for $\sigma>0$ consider $X_{\sigma}$, $\phi_{s,t}^{\sigma}$
given by \eqref {eq:def-ito-flow} and \eqref{eq:general-form-X-sigma} and let\[
\phi^{\sigma}:=\left\{ \phi_{s,t}^{\sigma}(\cdot)\,:\,0\leq s\leq t\leq T\right\} \]
be the forward stochastic \Ito flow of $C^{k}$-diffeomorphisms associated
to $X_{\sigma}$.

The following theorem was proved in slightly more generality (i.e.
replacing $\U$ by $\Rr^{n}$) in \cite{budhiraja2008large}.
\begin{thm}
\label{thm:BDM-3-2}\cite{budhiraja2008large} The family $(\phi^{\sigma},X_{\sigma})_{\sigma>0}$
satisfies a LDP in the spaces $\hat{W}_{k-1}\times W_{k-1}$ and $W_{k-1}\times W_{k-1}$
with rate function \[
I_{W}(\phi,X)=\inf\left\{ \frac{1}{2}\int_{0}^{T}\norm{f(s)}_{l_{2}}^{2}ds\,\,:\,\, f\in L^{2}(0,T;l_{2})\mbox{ s.t. }(\phi^{0,f},X^{0,f})=(\phi,X)\right\} \,.\]

\end{thm}
Below, in Theorem \ref{thm:LDP-Flow-2}, we generalize this theorem to $W_{k-1,id}^{\gamma}(\Q)\times W_{k-1,0}^{\gamma}(\Q)$ (see the definition in \eqref{eq:def-Wm-Rr-n2}).
However, we first need to show that $(\phi^{\sigma},X_{\sigma})_{\sigma>0}$
have enough regularity:
\begin{lem}
\label{lem:regularity-phi-X}For all $\sigma>0$ and all $0<\gamma<\frac{1}{2}$,
the pair $(\phi^{\sigma},X_{\sigma})$ is in $\hat{W}_{k-1}(\Q)\times W_{k-1,0}^{\gamma}(\Q)$
and $W_{k-1,id}^{\gamma}(\Q)\times W_{k-1,0}^{\gamma}(\Q)$ almost
surely.  \end{lem}
\begin{proof}
The proof follows \cite[Prop. 4.10]{budhiraja2008large}. Introducing
the notation $\norm{\cdot}_{j,p}$ for the norm on $W^{j,p}(\U)$, we note that according to \cite{budhiraja2008large} Lemmas 4.7-4.9, for each $p>1$ there
exists $C_{p}$ such that \begin{align*}
\sup_{\sigma}\Ee\norm{\phi_{t}^{\sigma}-\phi_{s}^{\sigma}}_{k,p}^{p} & \leq C_{p}\left|t-s\right|^{p/2}\\
\sup_{\sigma}\Ee\norm{X_{\sigma}(t,\cdot)-X_{\sigma}(s,\cdot)}_{k,p}^{p} & \leq C_{p}\left|t-s\right|^{p/2}\end{align*}
and due to the initial values $X_{\sigma}(0,\cdot)=0$, $\phi_{0}^{\sigma}(x)=x$
we also have \[
\sup_{\sigma}\Ee\norm{\phi_{t}^{\sigma}}_{k,p}^{p}+\sup_{\sigma}\Ee\norm{X_{\sigma}(t,\cdot)}_{k,p}^{p}\leq C_{p}\,.\]
Since the Sobolev embedding $W^{k,p}(\U)\hookrightarrow C^{k-1,\gamma}(\ooU)$
is continuous if $\gamma>0$ and $\frac{3}{p}<1-\gamma$ \cite{Adams1975}, also $W^{k,p}(\U)\hookrightarrow\lambda^{k-1,\gamma}(\ooU)$
is continuous for all $\gamma>0$ with $\frac{3}{p}<1-\gamma$. Lemma \ref{lem:Hairer-Maas-Weber}
yields the desired Hölder-regularity in time.The other properties, i.e. $\phi$ being a diffeomorphism and the boundary values follow from \cite{Kuni97}, Theorem 4.6.5 and Assumption \ref{ass:main-assumption}.  
\end{proof}

\subsection{Proof of Theorem \ref{thm:Main-Thm-1}} \label{sub:proof-Main-Thm-1}

Let $\left\{ f_{n}\right\} _{n\in\Nn}$ with $f_{n}=\left\{ f_{n}^{l}\right\} _{l\in\Nn}$
be a sequence in $\mathA_{N}[l_{2}]$ for some fixed $N<\infty$ and
let $f\in\mathA_{N}[l_{2}]$. Let $\left\{ \sigma_{n}\right\} _{n\in\Nn}$
be a sequence of positive numbers such that $\sigma_{n}\rightarrow0$
for $n\rightarrow\infty$. For simplicity of notation, we set \[
M(t,x):=\sum_{k=1}^{\infty}\int_{0}^{t}X^{(k)}(s,x)\, dB_{k}(s)\,,\qquad\forall(x,t)\in\Rr^{n}\times[0,T]\,.\]
Then, we define the following quantities: \begin{align}
X_{n}(t,x) & =\int_{0}^{t}b_{f_{n}}(s,x)\, ds+\sqrt{\sigma_{n}}\int_{0}^{t}M(ds,x)\,,\label{eq:Def-X-n}\\
X_{0}(t,x) & =\int_{0}^{t}b_{f}(s,x)\, ds\,,\nonumber \\
\phi_{t}^{n}(x) & =x+\int_{0}^{t}b_{f_{n}}(s,\phi_{s}^{n}(x))ds+\sqrt{\sigma_{n}}\int_{0}^{t}M(ds,\phi_{s}^{n}(x))\,,\label{eq:Def-phi-n}\\
\phi_{t}^{0}(x) & =x+\int_{0}^{t}b_{f}(s,\phi_{s}^{0}(x))ds\,\nonumber 
\end{align}
where $b_{f_n}$, $b_{f}$ are defined by \eqref{eq:LDP-X-0-f}.

Note that due to Lemma \ref{lem:regularity-phi-X}, $\left(\phi_{t}^{n},X_{n}\right),\left(\phi_{t}^{0},X_{0}\right)\in W_{k-1,id}^{\gamma}(\Q)\times W_{k-1,0}^{\gamma}(\Q)$.
We next specify suitable notions of weak convergence in $\hat{W}_{k-1}\times W_{k-1}$
and $W_{k-1,id}^{\gamma}(\Q)\times W_{k-1,0}^{\gamma}(\Q)$:
\begin{defn}
\cite{budhiraja2008large} Let $\hat{\Pp}_{k-1}^{n}$, $\hat{\Pp}_{k-1}^{0}$
be the measures induced by $\left(\phi_{t}^{n},X_{n}\right)$, $\left(\phi_{t}^{0},X_{0}\right)$
respectively on $\hat{W}_{k-1}\times W_{k-1}$, i.e. \[
\hat{\Pp}_{k-1}^{n}(A)=\Pp\left((\phi^{n},X_{n})\in A\right)\,,\quad\hat{\Pp}_{k-1}^{0}(A)=\Pp\left((\phi^{0},X_{0})\in A\right)\qquad\forall A\in\banachB(\hat{W}_{k-1}\times W_{k-1})\,.\]
The sequence $\left\{ \left(\phi_{t}^{n},X_{n}\right)\right\} _{n\in\Nn}$
is said to converge weakly as $G^{k-1}$-flows to $\left(\phi_{t}^{0},X_{0}\right)$
as $n\rightarrow\infty$ if $\hat{\Pp}_{k-1}^{n}$ converges weakly as measures
to $\hat{\Pp}_{k-1}^{0}$ as $n\rightarrow\infty$.
\end{defn}
\begin{defn}
Let $\Pp_{k-1}^{n}$, $\Pp_{k-1}^{0}$ be the measures induced by
$\left(\phi_{t}^{n},X_{n}\right)$, $\left(\phi_{t}^{0},X_{0}\right)$
respectively on $W_{k-1,id}^{\gamma}(\Q)\times W_{k-1,0}^{\gamma}(\Q)$. The sequence
$\left\{ \left(\phi_{t}^{n},X_{n}\right)\right\} _{n\in\Nn}$ is said
to converge weakly as $C_{\gamma}^{k-1}$-flows to $\left(\phi_{t}^{0},X_{0}\right)$
as $n\rightarrow\infty$ if $\Pp_{k-1}^{n}$ converges weakly as measures to $\Pp_{k-1}^{0}$
as $n\rightarrow\infty$.
\end{defn} 
Note that the last definition makes sense in view of Lemma \ref{lem:regularity-phi-X} 
which guaranties 
$\left(\phi_{t}^{n},X_{n}\right)\in W_{k-1,id}^{\gamma}(\Q)\times W_{k-1,0}^{\gamma}(\Q)$.
We find the following weak continuity property of the mapping $f\mapsto(\phi_t,X)$, 
which was proved in \cite{budhiraja2008large} for the $G^{k-1}$-case:
\begin{thm}
\label{thm:c-k-gamma-convergence}Let $f_{n}$ converge to $f$ in
distribution as $S_{N}[l_{2}]$-valued sequence of random variables.
Then the sequence $\left\{ \left(\phi^{n},X_{n}\right)\right\} _{n\in\Nn}$
converges weakly as $C_{\gamma}^{k-1}$-flows and $G^{k-1}$-flows
to the pair $\left(\phi_{t}^{0},X_{0}\right)$ as $n\rightarrow\infty$
for any $\gamma<\frac{1}{2}$.\end{thm}
We postpone the proof of this theorem to Section \ref{sub:Proof-of-Theorem-3-4}.

Let $\Rr^{\infty}:=\prod_{n\in\Nn}\Rr$ be the usual product space.
Then, $\mathS=C([0,T];\Rr^{\infty})$ is a Polish space and $\beta=\left\{ B_{i}\right\} _{i\in\Nn}$
is a random $\mathS$-valued variable (see \cite{budhiraja2008large}).
As shown in \cite{budhiraja2008large}, proofs of theorems in the 
spirit of Theorem \ref{thm:Main-Thm-1} or  
Theorem \ref{thm:LDP-Flow-2} below basically reduce to applications of 
the following result: 
\begin{thm}
\cite[Theorem 3.6]{budhiraja2008large}\label{thm:Budhiraj-abstract}
Let $\mathE$ be a Polish space, let $\left\{ \mathG^{\sigma}\right\} _{\sigma\geq0}$
be a collection of measurable maps from $(\mathS,\banachB(\mathS))$ to $(\mathE,\banachB(\mathE))$
and let $X^{\sigma}=\mathG^{\sigma}(\sqrt{\sigma}\beta)$. Suppose
that there exists a measurable map $\mathG^{0}:\mathS\rightarrow\mathE$
such that for every $N<\infty$ the set $\Gamma_{N}:=\left\{ \mathG^{0}(\int_{0}^{\cdot}u(s)\, ds)\,:\, u\in S_{N}[l_{2}]\right\} $
is a compact subset of $\mathE$. For $f\in\mathE$ let \[
\mathC_{f}:=\left\{ u\in L^{2}(0,T\,;\, l_{2})\,:\, f=\mathG^{0}\left(\int_{0}^{\cdot}u(s)\, ds\right)\right\} .\]
 Then, $\hat{I}$ defined by \[
\hat{I}(f)=\inf_{u\in\mathC_{f}}\left\{ \frac{1}{2}\int_{0}^{T}\norm{u(s)}_{l_{2}}^{2}ds\right\} \,,\qquad f\in\mathE\,,\]
is a good rate function on $\mathE$. Furthermore, suppose that for all
$N<\infty$ and families $\left\{ u^{\sigma}\right\} \subset\mathA_{N}[l_{2}]$
such that $u^{\sigma}$ converges in distribution to some $u\in\mathA_{N}[l_{2}]$,
we have that \[
\mathG^{\sigma}(\sqrt{\sigma}\beta+\int_{0}^{\cdot}u^{\sigma}(s)\, ds)\rightarrow\mathG^{0}(\int_{0}^{\cdot}u(s)\, ds)\]
in distribution as $\sigma\rightarrow0$. Then the family $\left\{ X^{\sigma}\,:\,\sigma>0\right\} $
satisfies the LDP on $\mathE$ with good rate function $\hat{I}$.
\end{thm}
In order to prove Theorem \ref{thm:Main-Thm-1}, we remark that the Stratonovic and the \Ito flows are related through
\begin{align*}
d\varphi^\sigma_{s,t}(x) & =-X_\sigma(\circ dt,\varphi_{s,t}(x))\\
 & =-\sqrt\sigma \sum_{i=1}^{\infty}X^{(i)}(t,\varphi_{s,t})\,\circ dB_{i}(t)-X^{(0)}(t,\varphi_{s,t})\, dt\\
 & =-\sqrt\sigma \sum_{i=1}^{\infty}X^{(i)}(t,\varphi_{s,t})\, dB_{i}(t)-X^{(0)}(t,\varphi_{s,t})\, dt\\
 &\phantom{=}-\frac{\sqrt\sigma}{2} \sum_{i=1}^{\infty}X^{(i)}(t,\varphi_{s,t})\partial_\varphi X^{(i)}(t,\varphi_{s,t})\, dt\\
&=-X_\sigma(dt,\varphi_{s,t}(x))-\frac{\sqrt\sigma}{2} \sum_{i=1}^{\infty}X^{(i)}(t,\varphi_{s,t})\partial_\varphi X^{(i)}(t,\varphi_{s,t})\, dt
\end{align*}
Thus, $\varphi_{s,t}$ and $\phi_{s,t}$ are exponentially equivalent 
in the space $\lambda^\gamma([0,T];C^{k-1,\beta}(\ooU))$ in the sense 
of \cite{dembo1998large}, Definition 4.2.10 and thus, according to 
\cite{dembo1998large} Theorem 4.2.13, it is enough to show the large 
deviation principle for $\phi^\sigma$. An application of Theorem 
\ref{pro:koenig} and the continuous embedding
$\lambda^{k,\alpha}\hookrightarrow C^{k,\alpha}$ yields that Theorem
\ref{thm:Main-Thm-1} is a direct consequence of the following Theorem:
\begin{thm}
\label{thm:LDP-Flow-2}For any $0<\gamma<\frac{1}{2}$, the family
$(\phi^{\sigma},X_{\sigma})_{\sigma>0}$ satisfies a LDP in the
spaces $\hat{W}_{k-1}(\Q)\times W_{k-1,0}^{\gamma}(\Q)$ and $W_{k-1,id}^{\gamma}(\Q)\times W_{k-1,0}^{\gamma}(\Q)$
with rate function $I_{W}^{\ast}(\phi,X)$ defined in \eqref{eq:LDP-rate-function}. 
\end{thm}

\begin{proof}
We will only show that the family $(\phi^{\sigma},X_{\sigma})_{\sigma>0}$
satisfies a LDP in the spaces $W_{k-1,id}^{\gamma}(\Q)\times W_{k-1,0}^{\gamma}(\Q)$
with rate function $I$ defined as in \eqref{eq:LDP-rate-function}.
The LDP in $\hat{W}_{k-1}\times W_{k-1,0}^{\gamma}(\Q)$ can be
shown in a similar way. The proof follows the proof of
Theorem \ref{thm:BDM-3-2} in \cite{budhiraja2008large}. Our aim
is to reduce the statement of Theorem \ref{thm:LDP-Flow-2} to an
application of Theorem \ref{thm:Budhiraj-abstract}. 

Let $\mathG^{\sigma}:\mathS\rightarrow W_{k-1,id}^{\gamma}(\Q)\times W_{k-1,0}^{\gamma}(\Q)$
be measureable such that $\mathG(\sqrt{\sigma}\beta)=(\phi^{\sigma},X_{\sigma})$ a.s.,
where $(\phi^{\sigma},X_{\sigma})$ are given through 
\eqref{eq:def-ito-flow} and \eqref{eq:general-form-X-sigma}.
Furthermore, we define $\mathG^{0}:\mathS\rightarrow W_{k-1,id}^{\gamma}(\Q)\times W_{k-1,0}^{\gamma}(\Q)$
by $\mathG^{0}(\int_{0}^{\cdot}f(s)ds)=(\phi^{0,f},X^{0,f})$ if
$f\in L^{2}(0,T;l_{2})$, where $(\phi^{0,f},X^{0,f})$ are defined
through \eqref{eq:LDP-X-0-f} and \eqref{eq:phi-0-f}. The mapping
$\mathG^{0}(\cdot)$ is extended by $0$ to the whole of $\mathS$.

In a first step, we consider the set $\Gamma_{N}:=\left\{ \mathG^{0}(\int_{0}^{\cdot}f(s)ds)\,:\, f\in S_{N}[l_{2}]\right\} $
for fixed $N\in\Nn$ and we show that $\Gamma_{N}$ is compact in
$W_{k-1,id}^{\gamma}(\Q)\times W_{k-1,0}^{\gamma}(\Q)$. Since
$S_{N}[l_{2}]$ with the weak topology is a polish space, we have
to show that $f_{n}\rightharpoonup f$ weakly in $S_{N}[l_{2}]$ implies
$\mathG^{0}(\int_{0}^{\cdot}f_{n}(s)ds)\rightarrow\mathG^{0}(\int_{0}^{\cdot}f(s)ds)$
strongly in $W_{k-1,id}^{\gamma}(\Q)\times W_{k-1,0}^{\gamma}(\Q)$.
If we set $\sigma_{n}=0$ in \eqref{eq:Def-X-n}--\eqref{eq:Def-phi-n},
we see that $(\phi^{n},X_{n})=\mathG^{0}(\int_{0}^{\cdot}f_{n}(s)ds)$
and Theorem \ref{thm:c-k-gamma-convergence} yields convergence $(\phi^{n},X_{n})\rightarrow\mathG^{0}(\int_{0}^{\cdot}f(s)ds)$.

Next, let $\left( f_{n}\right) \subset\mathA_{N}[l_{2}]$ converge
to $f\in\mathA_N[l_2]$ weakly in distribution and let $(\sigma_{n})_n$ 
be a sequence of positive numbers such that $\sigma_{n}\rightarrow0$
as $n\rightarrow\infty$. If we can show that \begin{equation}
\mathG^{\sigma_{n}}(\sqrt{\sigma_{n}}\beta+\int_{0}^{\cdot}f_{n}(s)ds)\rightarrow\mathG^{0}(\int_{0}^{\cdot}f(s)ds)\label{eq:conv-girsanov}\end{equation}
in distribution in $W_{k-1,id}^{\gamma}(\Q)\times W_{k-1,0}^{\gamma}(\Q)$ as $n\rightarrow\infty$,
we can apply Theorem \ref{thm:Budhiraj-abstract} in order to conclude
the proof. 

Girsanov's theorem yields that $\sqrt{\sigma_{n}}\beta+\int_{0}^{\cdot}f_{n}(s)ds$
is a Brownian motion w.r.t. our given probability measure and comparing \eqref{eq:general-form-X-sigma}
with \eqref{eq:Def-X-n}--\eqref{eq:Def-phi-n} we see that $\mathG^{\sigma_{n}}(\sqrt{\sigma_{n}}\beta+\int_{0}^{\cdot}f_{n}(s)ds)=(\phi^{n},X_{n})$
with $\left(\phi^{n},X_{n}\right)$ given through \eqref{eq:Def-X-n}--\eqref{eq:Def-phi-n}.
Also we remark once more that $\mathG^{0}(\int_{0}^{\cdot}f(s)ds)=(\phi^{0},X_{0})$
where $(\phi^{0},X_{0})$ are defined through \eqref{eq:Def-X-n}
and \eqref{eq:Def-phi-n}. The convergence \eqref{eq:conv-girsanov} now
follows from Theorem \ref{thm:c-k-gamma-convergence}.
\end{proof}

\subsection{\label{sub:Proof-of-Theorem-3-4}Proof of Theorem \ref{thm:c-k-gamma-convergence}}

The proof of Theorem \ref{thm:c-k-gamma-convergence} follows the
outline of the proof of Theorem 3.5 in \cite{budhiraja2008large}.
In particular, convergence as a $G^{k-1}$-flow was proved in \cite{budhiraja2008large}
and it is sufficient to show convergence as $C_{\gamma}^{k-1}$-flow.

We start with some preparations and generalize the notion of {}``convergence
as diffusions'' to the case of $C^{\gamma}$- regularity in time.
Like in \cite{budhiraja2008large}, let $\boldsymbol{x}=\left(x_{1},x_{2},\dots,x_{m}\right)$
and $\boldsymbol{y}=\left(y_{1},y_{2},\dots,y_{p}\right)$ be arbitrary
points in $\Rr^{d\times m}$ and $\Rr^{d\times p}$, respectively,
and set \begin{align*}
\phi_{t}^{n}(\boldsymbol{x}) & =\left(\phi_{t}^{n}(x_{1}),\phi_{t}^{n}(x_{2}),\dots\phi_{t}^{n}(x_{m})\right)\,,\\
X_{n}(\boldsymbol{y},t) & =\left(X_{n}(y_{1},t),X_{n}(y_{2},t),\dots,X_{n}(y_{p},t)\right)\,,\end{align*}
then $\left\{ \phi_{t}^{n}(\boldsymbol{x}),X_{n}(\boldsymbol{y},t)\right\} $
is a $C^{0,\gamma}$ stochastic process with values in $\Rr^{d\times m}\times\Rr^{d\times p}$
and is equally called $(m+p)$-point motion of the flow. Set $V_{m}^{\gamma}:=C^{0,\gamma}([0,T]\,;\,\Rr^{d\times m})$
and let $V_{m,p}^{\gamma}:=V_{m}^{\gamma}\times V_{p}^{\gamma}$.
\begin{defn}
Let $\gamma<\frac{1}{2}$ and $\Pp_{(\boldsymbol{x},\boldsymbol{y})}^{n}$,
$\Pp_{(\boldsymbol{x},\boldsymbol{y})}^{0}$ be the measures induced
by $\left(\phi^{n}(\boldsymbol{x}),X_{n}(\boldsymbol{y})\right)$
and $\left(\phi^{0}(\boldsymbol{x}),X_{0}(\boldsymbol{y})\right)$,
respectively, on $V_{m,p}^{\gamma}$. The sequence $\left\{ \left(\phi^{n},X_{n}\right)\right\} _{n\in\Nn}$
is said to converge weakly as $\gamma$-diffusions to $\left(\phi^{0},X_{0}\right)$
as $n\rightarrow\infty$ if $\Pp_{(\boldsymbol{x},\boldsymbol{y})}^{n}$
converges weakly to $\Pp_{(\boldsymbol{x},\boldsymbol{y})}^{0}$ as
$n\rightarrow\infty$ for each $(\boldsymbol{x},\boldsymbol{y})\in\Rr^{d\times m}\times\Rr^{d\times p}$,
and $m,p\in\Nn$.
\end{defn}
The next theorem gives a useful characterization of convergence as $C_{\gamma}^{k}$-flows.
\begin{thm}
\label{thm:Equiv-weak-diff}The family of probability measures $\Pp_{k-1}^{n}$
converges weakly to the probability measure $\Pp_{k-1}^{0}$ on $W_{k-1,id}^{\gamma}(\Q)\times W_{k-1,0}^{\gamma}(\Q)$
as $n\rightarrow\infty$ if and only if the following two conditions
hold:
\begin{enumerate}
\item the sequence $\left\{ \left(\phi^{n},X_{n}\right)\right\} _{n\in\Nn}$
converges weakly as $\gamma$-diffusions to $\left(\phi^{0},X_{0}\right)$
as $n\rightarrow\infty$,
\item the sequence $\left\{ \Pp_{k-1}^{n}\right\} _{n\in\Nn}$ is tight.
\end{enumerate}
\end{thm}
\begin{proof}
Clearly, if $\Pp_{k-1}^n\rightarrow\Pp_{k-1}$ weakly as measures (1) and (2) hold. We thus only have to show the inverse implication.

Since $\left\{ \Pp_{k-1}^{n}\right\} _{n\in\Nn}$ is tight, we find
convergence of a subsequence $\left\{ \Pp_{k-1}^{n_{m}}\right\} _{m\in\Nn}$
to a measure $\tilde{\Pp}$ on $W_{k-1,id}^{\gamma}(\Q)\times W_{k-1,0}^{\gamma}(\Q)$.
Since the embedding $W_{k-1,id}^{\gamma}(\Q)\times W_{k-1,0}^{\gamma}(\Q)\hookrightarrow\cV$ 
with $\cV:=\lambda^{0,\gamma}(0,T;C(\overline{U}))\times\lambda^{0,\gamma}(0,T;C(\overline{U}))$
is continuous, $\left\{ \Pp_{k-1}^{n_{m}}\right\} _{m\in\Nn}$ converges 
weakly as measures to $\tilde{\Pp}$ in $\cV$.

With the notation introduced in \cite{Kuni97} right before the 
statement of Theorem 1.4.5, setting 
$S=\lambda^{0,\gamma}(0,T)^{2}$ and $\mathbb{I}=\overline{U}$ 
we can apply Theorem 1.4.5 of \cite{Kuni97} to get convergence of
$\left\{ \Pp_{k-1}^{n_{m}}\right\} _{m\in\Nn}$ to $\Pp_{k-1}^{0}$
on $\cV$, thus $\Pp_{k-1}^{0}=\tilde{\Pp}$. Since this identification holds
for any converging subsequence, the theorem is proved. 
\end{proof}
Thus, it remains to prove that $\Pp_{k-1}^{n}$ satisfies (1) and (2) 
of Theorem \ref{thm:Equiv-weak-diff}. We start with a proof of (2).
\begin{lem}
\label{lem:tightness} The sequence $\left\{ \Pp_{k-1}^{n}\right\} _{n\in\Nn}$
is tight.\end{lem}
\begin{proof}
With the notation of Lemma \ref{lem:regularity-phi-X}, for each $p>1$
there exists $C_{p}$ such that \begin{align*}
\sup_{n}\Ee\norm{\phi_{t}^{n}-\phi_{s}^{n}}_{k,p}^{p} & \leq C_{p}\left|t-s\right|^{p/2}\\
\sup_{n}\Ee\norm{X_{n}(t,\cdot)-X_{n}(s,\cdot)}_{k,p}^{p} & \leq C_{p}\left|t-s\right|^{p/2}
\end{align*}
\[\sup_{n}\Ee\norm{\phi_{t}^{n}}_{k,p}^{p}+\sup_{n}\Ee\norm{X_{n}(t,\cdot)}_{k,p}^{p}\leq C_{p}\,.\]
By the compact embedding $W^{k,p}(\U)\hookrightarrow\lambda^{k-1,2\gamma}(\U)$
for large $p$, applying Theorem \ref{thm:Kolmogorov-tightness} yields
tightness of $\left\{ \left(\phi^{n},X_{n}\right)\right\} _{n\in\Nn}$
in $W_{k-1,id}^{\gamma}(\Q)\times W_{k-1,0}^{\gamma}(\Q)$.
\end{proof}
The first condition of Theorem \ref{thm:Equiv-weak-diff} will be
verified in the following three Lemmas. 
\begin{lem}
\label{lem:gamma-cont-discrete}For all $\gamma<\frac{1}{2}$, $x,y\in\U$
\begin{equation}
\Ee\norm{\int_{0}^{t}M(ds,y)}_{C^{0,\gamma}([0,T])}+\Ee\norm{\int_{0}^{t}M(ds,\phi_{s}^{n}(x))}_{C^{0,\gamma}([0,T])}\leq C(\gamma)\,,\label{eq:lem-gamma-cont-discr}\end{equation}
where $C(\gamma)$ does not depend on $x,\, y$ or $n$.\end{lem}
\begin{proof}
By the Burkholder-Davis-Gundy inequality we find for any $p,q>1$:
\begin{align*}
\Ee\left|\int_{t_{0}}^{t_{1}}M(ds,\phi_{s}^{n}(x))\right|^{p} & \leq\Ee\left|\sup_{t_{0}\leq t\leq t_{1}}\int_{t_{0}}^{t}M(ds,\phi_{s}^{n}(x))\right|^{p}\\
 & \stackrel{\scriptstyle\eqref{eq:est-X-sqare-a}}{\leq} C_{p}\Ee\left|\sum_{k=1}^{\infty}\int_{t_{0}}^{t_{1}}\left|X^{(k)}\right|^2(s,\phi_{s}^{n}(x))\, ds\right|^{\frac{p}{2}}\\
 & \leq C_{p,q}\left(\norm a_{C([0,T];C^{k,\alpha}(\ooU))}\right)^{p/2}(t_{1}-t_{0})^{\frac{p}{2q}}\,,\end{align*}
where we used $\left|\sup_{t}\left|a(t)\right|^{q}\right|^{\frac{1}{q}}=\sup_{t}\left|a(t)\right|$
and $a\in C([0,T];C^{0,\alpha}(\ooU))$. From Lemma \ref{lem:Hairer-Maas-Weber},
we get $t\mapsto\int_{0}^{t}M(ds,\phi_{s}^{n}(x))\in C^{0,\gamma}([0,T])$
a.s. for all $\gamma<\frac{1}{2q}-\frac{1}{p}$ and since $p>1$ and
$q>1$ were arbitrary, the claim follows.\end{proof}
\begin{lem}
\label{lem:tightness-fixed-x}For each $x\in\Rr^{d}$ and each $\gamma<\frac{1}{2}$,
the sequence $\{h_n\}_{n\in\N}$, $h_n(t)=\left(\phi_{t}^{n}(x),X_{n}(t,x)\right)$
is tight in $C^{\gamma}([0,T]\,;\,\Rr^{d}\times\Rr^{d})$.\end{lem}
\begin{proof}
We will only show tightness of $\phi^{n}(x)$. By Chebyshev's inequality,
Lemma \ref{lem:gamma-cont-discrete} yields \[
\lim_{n\rightarrow\infty}P\left(\norm{t\mapsto\sqrt{\sigma_{n}}\int_{0}^{t}M(ds,\phi_{s}^{n}(x))}_{C^{0,\gamma}}>\eps\right)=0\qquad\forall\gamma<\frac{1}{2}\,,\quad\forall\eps>0\]
and by the compact imbedding $C^{0,\gamma_{1}}\hookrightarrow C^{0,\gamma_{2}}$
for $\gamma_{1}>\gamma_{2}$, the sequence 
$\left(t\mapsto\sqrt{\sigma_{n}}\int_{0}^{t}M(ds,\phi_{s}^{n}(x))\right)$
is tight in $C^{\gamma}([0,T];\Rr^{d})$ for all $\gamma<\frac{1}{2}$.
Thus, it remains to show tightness of the sequence 
$t\mapsto\left( \int_{0}^{t}b_{f_{n}}(s,\phi_{s}^{n}(x))ds\right) _{n\in\Nn}$.
First, due to the uniform bound on $f_{n}$ and the inequality
(3.4) of \cite{budhiraja2008large}, we find \[
\int_{0}^{T}\left|b_{f_{n}}(s,\phi_{s}^{n}(x))\right|^{2}ds\leq C\left(\norm a_{C([0,T];C^{0,\alpha}(U))}+\norm{X^{(0)}}_{C([0,T];C^{0,\alpha}(U))}^{2}\right)\,.\]
Thus an application
of Hölder's inequality yields\[
\Ee\left|\int_{t_{1}}^{t_{2}}b_{f_{n}}(s,\phi_{s}^{n}(x))\, ds\right|^{p}\leq\Ee\left|\int_{t_{1}}^{t_{2}}\left|b_{f_{n}}(s,\phi_{s}^{n}(x))\right|^{2}ds\right|^{\frac{p}{2}}(t_{1}-t_{2})^{\frac{p}{2}}\,.\]
From the last two inequalities in combination with Theorem \ref{thm:Kolmogorov-tightness},
we get that $t\mapsto\int_{0}^{t}b_{f_{n}}(s,\phi_{s}^{n}(x))ds$
is tight in $\lambda^{0,\gamma}([0,T];\Rr^{d}\times\Rr^{d})$ and
thus also in $C^{0,\gamma}([0,T];\Rr^{d}\times\Rr^{d})$ for all $\gamma<\frac{1}{2}$. \end{proof}
\begin{lem}
\label{lem:weak-conv-diff}Assume $f_{n}\rightarrow f$ in distribution as $S_{N}[l_{2}]$-valued
random variables. Then the sequence $\left\{ \left(\phi^{n},X_{n}\right)\right\} _{n\in\Nn}$
converges weakly as $\gamma$-diffusions to $\left(\phi^{0},X_{0}\right)$
as $n\rightarrow\infty$.\end{lem}
\begin{proof}
By Lemma \ref{lem:tightness-fixed-x}, the sequence $\left(\phi^n_\cdot(x),X_n(\cdot,x)\right)$
has a weak limit $(\bar\phi,\bar X)$. As shown in \cite{budhiraja2008large}, Prop. 4.6, the mapping 
\[
C([0,T];\Rr^{d})\times S_{N}[l_{2}]\rightarrow\Rr^{d}\,,\qquad(\xi,v)\mapsto\int_{0}^{t}b_{v}(s,\xi_{s})\, ds\]
is continuous. Thus, in the sense of $\hat{\Pp}_{k-1}^{n}$
(i.e. in sense of $\hat{W}_{k-1}\times W_{k-1}$) any weak limit point
$(\bar{\phi},\bar{X},\bar{f})$ of the sequence $(\phi^{n},X_{n},f_{n})$
satisfies for fixed $(t,x)$:\[
\bar{X}(x,t)=\int_{0}^{t}b_{\bar{f}}(s,x)\, ds\,,\qquad\bar{\phi}_{t}(x)=x+\int_{0}^{t}b_{\bar{f}}(s,\bar{\phi}_{s})\, ds\,\qquad\mbox{a.s..}\]

\end{proof}

\section{\label{sec:section-5}Large deviation principle for the stochastic
Allen-Cahn equation}

We will apply the large deviation principle for stochastic flows 
in the form of Theorem \ref{thm:LDP-Flow-2} to show a large deviation principle for the stochstic Allen-Cahn equation 
\eqref{eq:main-EQ-1}-\eqref{eq:main-EQ-2}.

In view of Theorem \ref{pro:koenig}, \eqref{eq:main-EQ-1}-\eqref{eq:main-EQ-2} and equation \eqref{eq:def-stratonovich-flow} it remains to show that
\begin{enumerate}
\item The mapping $A:\,\,\varphi\mapsto u$, where 
$u$ denotes the solution to 
\begin{equation}
\begin{aligned}u(t,x) & =u_{0}(x)+\int_{0}^{t}\left(\Delta u-W'(u)\right)+\int_{0}^{t}\nabla u(s,x)\circ d\varphi(s,\varphi^{-1}(s,x))\,,\\
\nabla u\cdot\nu_{\U} & =0\qquad\mbox{on }(0,T)\times\partial\U\end{aligned}
\label{eq:deterministic}\end{equation}
is well defined and continuous in an appropriate sense.
\item The good rate function $\tilde{I}(u)=\inf\left\{ I_{W}^{\ast}(\varphi)\,:\, u=A(\varphi)\right\} $
for the sequence $\left\{ u_{\sigma}\right\} _{\sigma>0}$ can be
written as in \eqref{eq:rate-func-AC}.
\end{enumerate}

In what follows, we assume that $\varphi\in W_{2,id}^{\alpha}(\Q)$ is a stochastic flow
and consider the SPDE \eqref{eq:main-EQ-1}--\eqref{eq:main-EQ-2} in the form \eqref{eq:deterministic}.
Supposing $\varphi$ had enough regularity in time, 
equation \eqref{eq:deterministic} would be equivalent to 
\[
\partial_{t}u+\nabla u\cdot\partial_{t}\varphi_{0,t}(\cdot)|_{\varphi_{0,t}^{-1}}-\Delta u+W'(u)=0\,,
\]
and we could use standard methods in partial differential equations.
However, as $\varphi_{0,t}$ is only $C^{0,\alpha}$ in time, $\alpha<\frac{1}{2}$,
we cannot interprete \eqref{eq:deterministic} in this form. 

Instead, we follow \cite{RoeWe13} and use that $u$ is a solution
of \eqref{eq:deterministic} if and only if the function $w$, given
by the transformation $w(t,x)=u(t,\varphi_{t}^{-1}(x))$ is a solution
of the following PDE:\begin{equation}
\partial_{t}w-R_{\varphi}:D^{2}w-S_{\varphi}\cdot\nabla w-w+w^{3}=0\,,\label{eq:AC-transform}\end{equation}
with initial condition $w(0,\cdot)=w_{0}(\cdot)=u_{0}(\cdot)$ and
boundary condition $\nabla w\cdot\nu_{\U}=\left(D\varphi\nabla u\right)\cdot\nu_{\U}=\nabla u\cdot\nu_{\U}=0$
and where the coefficients are given by
\begin{eqnarray*}
R_{\varphi}^{ij} & = & \sum_{k}\partial_{k}(\varphi_{0,t}^{-1})^{i}\circ\varphi_{0,t}\,\partial_{k}(\varphi_{0,t}^{-1})^{j}\circ\varphi_{0,t}\,,\\
S_{\varphi}^{i} & = & \sum_{k}\partial_{k}^{2}(\varphi_{0,t}^{-1})^{i}\circ\varphi_{0,t}\,.
\end{eqnarray*}
We find that $R_{\varphi}^{ij}$ is uniformly positive definit, with
$R_{\varphi}^{ij}\in C^{0,\alpha}([0,T];C^{2,2\alpha}(\ooU))$ and $S_{\varphi}^{i}\in C^{0,\alpha}([0,T];C^{1,2\alpha}(\ooU))$.
Note that as $\varphi(t,\cdot)|_{\Rr^{n}\backslash\U}\equiv Id$ for
all $t$, we find $R|_{\partial\U}=Id$ and $S|_{\partial\U}=0$.
Unique solvability of this problem in $C^{0,\alpha}(0,T;C^{3,\alpha}(\ooU))$
was shown in the proof of Theorem 4.1 in \cite{RoeWe13}.

\subsection{\label{sub:A-continuity-result}A continuity result}

In this subsection, we assume that $\varphi\in W_{2,id}^{\alpha}(\Q)$ is a deterministic function and that for given $\varphi$ the function $w$ is a solution to \eqref{eq:AC-transform}, which exists by \cite{RoeWe13}. It is the aim of this subsection to show that the mapping 
\begin{align*}
\mathcal{A}:\,\, W_{2,id}^{\alpha}(\Q)&\rightarrow H^{\frac{2+\beta}{2},2+\beta}(\Q)\\
\varphi&\mapsto w\,,\mbox{ where }w\mbox{ solves }\eqref{eq:AC-transform}\text{ with the above initial and boundary data}
\end{align*}
is continuous for any $\beta<\alpha$. This will imply continuity
of the mapping 
\begin{align}
\mathcal{B}:\,\, W_{2,id}^{\alpha}(\Q)&\rightarrow C([0,T];C^{2,\beta}(\ooU))\cap C^{0,\beta}([0,T];C^{1,\beta}(\ooU))\\
\varphi&\mapsto u(t,x):=\mathcal{A}(\varphi)(t,\varphi_{t}(x))=w(t,\varphi_{t}(x))\,.
\end{align}

To this aim, let $\varphi_{n}\rightarrow\varphi$ in $W_{2,id}^{\alpha}(\Q)$
and let $w:=\mathcal{A}(\varphi)$, $w_{n}:=\mathcal{A}(\varphi_{n})$,
as well as $R^{ij}:=R_{\varphi}^{ij}$, $S^{i}:=S_{\varphi}^{i}$,
$R_{n}^{ij}:=R_{\varphi_{n}}^{ij}$ and $S_{n}^{i}:=S_{\varphi_{n}}^{i}$. 
\begin{lem}
\label{lem:bound-w}We find $\norm{w_n}_{\infty}\leq\max\left\{ 1,\norm{u_{0}}_{\infty}\right\} $
and $w_{n}$ admits a uniform bound of the form \begin{equation}
\left|w_{n}\right|_{\Q,2+\alpha}\leq C\,.\label{eq:est-w-n}\end{equation}
\end{lem}
\begin{proof}
As $w_{n}$ solves \[
\partial_{t}w_{n}-R_{n}:D^{2}w_{n}-S_{n}\cdot\nabla w_{n}+\left(w_{n}^{3}-w_{n}\right)=0\,,\]
we multiply the last equation by $(w_{n}-c)^{+}$ for some $c>\max\left\{ 1,\norm{u_{0}}_{\infty}\right\} $
and integrate over $\U$ to find \begin{multline*}
\frac{1}{2}\frac{d}{dt}\int_{\U}\left(\left(w_{n}-c\right)^{+}\right)^{2}-\int_{\U}\left(w_{n}-c\right)^{+}R_{n}:D^{2}w_{n}-\int_{\U}\left(w_{n}-c\right)^{+}S_{n}\cdot\nabla\left(w_{n}-c\right)^{+}\\
+\int_{\U}\left(w_{n}^{3}-w_{n}\right)\left(w_{n}-c\right)^{+}=0\,.\end{multline*}
Since $c\geq1$, we find $\left(w_{n}^{3}-w_{n}\right)\left(w_{n}-c\right)^{+}\geq0$.
Thus, integration by parts in the second integral on the left hand
side with the boundary condition $\nabla w_{n}\cdot\nu=(R_{n}\nabla w)\cdot\nu=0$
yields\begin{multline*}
\frac{1}{2}\frac{d}{dt}\int_{\U}\left(\left(w_{n}-c\right)^{+}\right)^{2}+\int_{\U}\nabla\left(w_{n}-c\right)^{+}R_{n}:\nabla\left(w_{n}-c\right)^{+}\\
\leq\int_{\U}\left(w_{n}-c\right)^{+}\left(S_{n}-\diver R_{n}\right)\cdot\nabla\left(w_{n}-c\right)^{+}\,.\end{multline*}
Ellipticity of $R_{n}$ and essential boundedness of $R_{n}$, $S_n$ and
$\diver R_{n}$ yield\[
\frac{1}{2}\frac{d}{dt}\int_{\U}\left(\left(w_{n}-c\right)^{+}\right)^{2}+C_{n}\int_{\U}\left|\nabla\left(w_{n}-c\right)^{+}\right|^{2}\leq\int_{\U}\left(\left(w_{n}-c\right)^{+}\right)^2\,.\]
Gronwall's inequality then yields
\[
\frac{1}{2}\sup_{t\in[0,T]}\int_{\U}\left(\left(w_{n}-c\right)^{+}(t)\right)^{2}+\int_{0}^{T}C_{n}\int_{\U}\left|\nabla\left(w_{n}-c\right)^{+}\right|^{2}\leq\int_{0}^{T}\int_{\U}\left(\left(w_{n}-c\right)^{+}\right)^{2}\,.
\]
and thus $w_{n}\leq c$ (see also \cite{evans1998partial} Chapter
7). Similarly, we get $-c\leq w_{n}$. 

We only have to show \eqref{eq:est-w-n} for $n>n_{0}$ with $n_{0}\in\Nn$.
Note that $w_{n}$ equally solves\[
\partial_{t}w_{n}-R:D^{2}w_{n}-S\cdot\nabla w_{n}=f_{n}\,,\]
where \[
f_{n}=(R-R_{n}):D^{2}w_{n}+(S-S_{n})\cdot\nabla w_{n}-\left(w_{n}^{3}-w_{n}\right)\,.\]
Thus, by Schauder estimates (\cite{ladyzenskaya1968quasilinear} Theorem
IV.5.3) we find \begin{equation}
c\left|w_{n}\right|_{\Q,2+\alpha}\leq(\left|f_{n}\right|_{\Q,\alpha}+\left|u_{0}\right|_{\U,2+\alpha})\label{eq:schauder-est}\end{equation}
with \[
\left|f_{n}\right|_{\Q,\alpha}\leq C\left(\left|\varphi_{n}-\varphi\right|_{W_{2,id}^{\alpha}}\left(\left|D^{2}w_{n}\right|_{\Q,\alpha}+\left|\nabla w_{n}\right|_{\Q,\alpha}\right)+\left(\norm{w_{n}}_{\infty}^{2}+1\right)\left|w_{n}\right|_{\Q,\alpha}\right)\]
By Ehrlings lemma we find $\left|w_{n}\right|_{\Q,\alpha}\leq\delta\left|w_{n}\right|_{\Q,2+\alpha}+C_{\delta}\norm{w_{n}}_{\infty}$ and deduce 
\[
\left|f_{n}\right|_{\Q,\alpha}\leq C\left(\left|\varphi_{n}-\varphi\right|_{W_{2,id}^{\alpha}}\left|w_{n}\right|_{\Q,2+\alpha}+\left(\norm{w_{n}}_{\infty}^{2}+1\right)\left(\delta\left|w_{n}\right|_{\Q,2+\alpha}+C_{\delta}\norm{w_{n}}_{\infty}\right)\right)\,.
\]
We use this estimate and $\norm{w_{n}}_{\infty}^{2}\leq c^{2}$ in
\eqref{eq:schauder-est} to obtain\[
\left|w_{n}\right|_{\Q,2+\alpha}\leq C\left(\left(\left|\varphi_{n}-\varphi\right|_{W_{2,id}^{\alpha}}+\delta\left(c^{2}+1\right)\right)\left|w_{n}\right|_{\Q,2+\alpha}+c\left(c^{2}+1\right)C_{\delta}+\left|u_{0}\right|_{\U,2+\alpha}\right)\,.\]
Since $\varphi_{n}\rightarrow\varphi$ in $W_{2,id}^{\alpha}(\Q)$,
we can choose $n$ large enough and $\delta$ small enough such that
$\left(\left|\varphi_{n}-\varphi\right|_{W_{2,id}^{\alpha}}+\delta\left(c^{2}+1\right)\right)<\frac{1}{2C}$.
Absorbing $\frac{1}{2}\left|w_{n}\right|_{\Q,2+\alpha}$ on the left
hand side, we finally obtain the
desired estimate on $\left|w_{n}\right|_{\Q,2+\alpha}$ for all $n>N_{0}$
with $N_{0}\in\Nn$ fixed. Thus, we obtain \eqref{eq:est-w-n} for
all $n$.
\end{proof}
Using this boundedness result, we get continuity of the mapping $\mathcal{A}$:
\begin{lem}
\label{lem:continuity}Let $\varphi_{n}\rightarrow\varphi$ in $W_{2,id}^{\alpha}(\U)$
and let $w:=\mathcal{A}(\varphi)$, $w_{n}:=\mathcal{A}(\varphi_{n})$.
We find $w_{n}\rightarrow w$ in $H^{\frac{2+\beta}{2},2+\beta}(\Q)$
for all $\beta<\alpha$.\end{lem}
\begin{proof}
We define $\tilde{w}_{n}:=w-w_{n}$ and find that $\tilde{w}_{n}\in H^{\frac{2+\beta}{2},2+\beta}(\Q)$
is the unique solution (see \cite[Thm IV.5.3]{ladyzenskaya1968quasilinear})
to 
\begin{equation}
\partial_{t}\tilde{w}_{n}-R_{n}:D^{2}\tilde{w}_{n}-S_{n}\cdot\nabla\tilde{w}_{n}+\tilde{w}(w^{2}+ww_{n}+w_{n}^{2})=f_{n}+\tilde{w}_{n}\,,\label{eq:w-tilde-stark}
\end{equation}
where we assume $w_{n}$, $w$ and \[
f_{n}=(R_{n}-R):D^{2}w+(S_{n}-S)\cdot\nabla w\,\]
to be fixed. Since $R_{n}\rightarrow R$ in $C(\Q)$ and $R$ is elliptic,
the sequence $R_{n}$ is uniformly elliptic for $n$ large enough.
Equation \eqref{eq:w-tilde-stark} multiplied with $\tilde{w}_{n}$
and integrated over $\U$ gives 
\begin{multline*}
\frac12\frac{d}{dt}\int_{\U}\tilde{w}_{n}^{2}+\int_{\U}\nabla\tilde{w}_nR_n\nabla\tilde{w}_{n}+\int_{\U}\tilde{w}_n^2(w^{2}+ww_{n}+w_{n}^{2}) \\
\leq \int_{\U}f_{n}\tilde w_n+\int_{\U}\left(\tilde w_nS_n\nabla\tilde w_n-\tilde w_n\diver R_n\cdot\nabla\tilde w_n\right)\,,
\end{multline*}
where we used that $(R_n\nabla\tilde w_n)\cdot\nu_{\U}=0$ on $\partial\U$. Using positivity of $w^2+ww_n+w_n^2$ as well
as uniform boundedness of $D\, R_{n}$ and $S_{n}$ and uniform ellipticity
of $R_{n}$ we get from the last equation
\begin{equation}
\frac{d}{dt}\int_{\U}\tilde{w}_{n}^{2}+c\int_{\U}\left|\nabla\tilde{w}_{n}\right|^{2}\leq C\left(\int_{\U}f_{n}^{2}+\int_{\U}\tilde{w}_{n}^{2}\right)\,.\label{eq:lem-2-14-est-1}
\end{equation}
Since $(R_{n}-R)\rightarrow0$ and $(S_{n}-S)\rightarrow0$ uniformly in $[0,T]\times\overline U$ 
we obtain $\int_{\U}f_{n}^{2}(t)\rightarrow0$ for all $t\in[0,T]$. Since $\left|w_{n}\right|_{\Q,2+\alpha}$
is bounded by Lemma \ref{lem:bound-w}, Gronwall's inequality applied
to \eqref{eq:lem-2-14-est-1} yields \[
\sup_{0\leq t\leq T}\norm{\tilde{w}_{n}(t)}_{L^{2}(\U)}^{2}+\int_{0}^{T}\norm{\nabla\tilde{w}_{n}}_{L^{2}(\U)}^{2}\rightarrow0\quad\mbox{as}\quad n\rightarrow\infty\,,\]
and thus $\tilde w_n(t,x)\rightarrow 0$ for a.e. $(t,x)\in(0,T)\times\U$. Estimate \eqref{eq:est-w-n} yields uniform boundedness of $\tilde{w}_{n}$
in $H^{\frac{2+\beta}{2},2+\beta}$ for all $\beta<\alpha$ and by
Remark \ref{rem:hoelder-imbeddings} compactness in $H^{\frac{2+\beta}{2},2+\beta}$
for all $\beta<\alpha$. The pointwise convergence a.e. to $0$ and compactness
of $\tilde{w}_{n}$ yield convergence of $\tilde{w}_{n}$ to $0$
in $H^{\frac{2+\beta}{2},2+\beta}$ for all $\beta<\alpha$.
\end{proof}

\subsection{\label{sub:Proof-of-Theorem-4}Proof of Theorem \ref{thm:Main-1}}

We now turn to the proof of Theorem \ref{thm:Main-1}. From Proposition
\ref{thm:LDP-Flow-2}, Theorem \ref{pro:koenig} and Lemma \ref{lem:continuity},
it follows that $u_{\sigma}$ satisfies an LDP in $\sigma$ with rate
function\[
\tilde{I}(u)=\inf\left\{ I(\varphi)\,:\,\varphi\in W_{2,id}^{\alpha}(\Q),\,\,\mbox{ s.t. \eqref{eq:deterministic} holds for }u\mbox{ and }\varphi\,\right\} \]
Thus, it is sufficient to prove that \[
\hat{I}(u)=\tilde{I}(u)\,.\]
Fix $u\in C([0,T];C^{2,\beta}(\ooU))\cap C^{0,\beta}([0,T];C^{1,\beta}(\ooU))$, let $\delta>0$ and
choose $\varphi\in W^\alpha_{2,id}(\Q)$ with $I(\varphi)\leq\tilde{I}(u)+\frac{\delta}{2}$ 
such that\eqref{eq:deterministic} holds. Then, there is $f\in L^{2}(0,T;l_{2})\mbox{ s.t. }\phi^{0,f}=\varphi$
and \[
\frac{1}{2}\int_{0}^{T}\norm{f(s)}_{l_{2}}^{2}ds\leq I(\varphi)+\frac{\delta}{2}\leq\tilde{I}(u)+\delta\,.\]
Furthermore, for such $f$ holds 
\begin{equation}
u(t)=u_{0}(\cdot)+\int_{0}^{t}\nabla u\cdot\left(\sum_{l=1}^{\infty}f_{l}(t)X^{(l)}(x,t)+X^{(0)}(s,x)\right)+\int_{0}^{t}\left(\Delta u-W'(u)\right)=0\,.\label{eq:sol-to-f}
\end{equation}
Therefore $\hat{I}(u)\leq\tilde{I}(u)+\delta$ for all $\delta>0$,
and thus \[
\hat{I}(u)\leq\tilde{I}(u)\,.\]
On the other hand, for $\delta>0$ arbitrary, let $f\in L^{2}(0,T;l_{2})$ be such that \eqref{eq:sol-to-f}
holds and such that\[
\frac{1}{2}\int_{0}^{T}\norm{f(s)}_{l_{2}}^{2}ds\leq\hat{I}(u)+\frac{\delta}{2}\,.\]
Then we note that $\phi^{0,f}$, given by \eqref{eq:phi-0-f}, belongs to $C^{0,\alpha}([0,T];C^{3}(\U))$ with
$I(\phi^{0,f})\leq\frac{1}{2}\int_{0}^{T}\norm{f(s)}_{l_{2}}^{2}ds$.
Comparing \eqref{eq:sol-to-f} with \eqref{eq:deterministic} we get
\[
\tilde{I}(u)\leq I(\phi^{0,f})\leq\frac{1}{2}\int_{0}^{T}\norm{f(s)}_{l_{2}}^{2}ds\leq\hat{I}(u)+\frac{\delta}{2}\,.\]
Since $\delta>0$ was arbitrary, this shows $\hat{I}(u)=\tilde{I}(u)$
and concludes the proof.

\def\cprime{$'$}

\end{document}